\documentclass[11pt,a4paper,twoside]{article}

\usepackage{amsfonts}
\usepackage{mathrsfs}
\usepackage{amsthm}

\usepackage{amsmath}%
\usepackage{amssymb}%
\usepackage{fancyhdr}%
\usepackage{ifthen}%
\usepackage{calc}%
\usepackage{lastpage}%
\usepackage{array}%
\usepackage{anysize}%
\usepackage{cases}

\makeatletter
\renewcommand\@biblabel[1]{${#1}.$}
\makeatother

\title{\bf A  Probabilistic Approach for  Gradient Estimates on Time-Inhomogeneous Manifolds }
\author{{\bf Li-Juan Cheng \footnote{\scriptsize School of Mathematical Sciences,  Beijing Normal
University,  Laboratory of Mathematics and Complex
Systems,  Ministry of Education,  Beijing 100875,
The People's Republic of China.  E-mail: chenglj@mail.bnu.edu.cn(L.J. Cheng)
}}}
\date{June 12, 2013}

\newtheorem{theorem}{Theorem}[section]

\newtheorem{lemma}[theorem]{Lemma}

\newtheorem{remark}[theorem]{Remark}
\newtheorem{example}[theorem]{Example}

\numberwithin{equation}{section} \catcode`@=11

\begin{document}
\maketitle
\noindent{\bf Abstract}\ \
Gradient inequalities of the Hamilton type and the Li-Yau  type for positive solutions to the heat equation are established from a probabilistic viewpoint, which simplifies the proofs of some results of  Sun [{\it Pacific J. Math.}, 253 (2011), pp. 489--510].
\vskip12pt




\def\dint{\displaystyle\int}
\def\ct{\cite}
\def\lb{\label}
\def\ex{Example}
\def\vd{\mathrm{d}}
\def\dis{\displaystyle}
\def\fin{\hfill$\square$}
\def\thm{theorem}
\def\bthm{\begin{theorem}}
\def\ethm{\end{theorem}}
\def\blem{\begin{lemma}}
\def\elem{\end{lemma}}
\def\brem{\begin{remark}}
\def\erem{\end{remark}}
\def\bexm{\begin{example}}
\def\eexm{\end{example}}
\def\bcor{\bg{corollary}}
\def\ecor{\end{corollary}}
\def\r{\right}
\def\l{\left}
\def\var{\text {\rm Var}}
\def\lmd{\lambda}
\def\alp{\alpha}
\def\gm{\gamma}
\def\Gm{\Gamma}
\def\e{\operatorname{e}}
\def\gap{\text{\rm gap}}
\def\dsum{\displaystyle\sum}
\def\dsup{\displaystyle\sup}
\def\dlim{\displaystyle\lim}
\def\dlimsup{\displaystyle\limsup}
\def\dmax{\displaystyle\max}
\def\dmin{\displaystyle\min}
\def\dinf{\displaystyle\inf}
\def\be{\begin{equation}}
\def\de{\end{equation}}
\def\dint{\displaystyle\int}
\def\dfrac{\displaystyle\frac}
\def\zm{\noindent{\bf  Proof.\ }}
\def\endzm{\quad $\Box$}
\def\mO{\mathcal{O}}
\def\mW{\mathcal{W}}
\def\mL{\mathcal{L}}
\def\LC{\mathcal{L}{\rm Cut}}
\def\proclaim#1{\bigskip\noindent{\bf #1}\bgroup\it\  }
\def\endproclaim{\egroup\par\bigskip}
\baselineskip 18pt

\section{Introduction}
\hspace{0.5cm} In this paper, we want to establish   gradient estimates on time-inhomogeneous manifolds
by stochastic analysis.  It is well known that the gradient estimates, known as the differential Harnack
inequalities, are powerful tools both on geometry and stochastic analysis. For example, R. Hamilton \cite{Ha95}
established differential Harnack inequalities for the scalar curvature under the Ricci flow, which is applied to the singularity analysis; Perelman \cite{Pe02} successfully used a differential Harnack inequalities to consider the Poincar\'{e} conjecture.

   Under some curvature constraints, Sun \cite{Sun11} gave gradient inequalities for positive solutions to  the heat equation under general geometric flows. Meanwhile,  Bailesteanu-Cao-Pulemotov \cite{BCP}, Liu \cite{Liu09} independently consider these similar problems under the Ricci flow. As announced, we want to review these inequalities from a probabilistic viewpoint. When the metric is independent of  $t$, this point of view has been
worked well for local estimates  in  positive harmonic function \cite{ADT}  and  for Li-Yau type gradient estimates \cite{AT} on Riemannian manifolds.

 Let $M$ be a $d$-dimensional  differential manifold without boundary  equipped with  a family of complete
Riemannian metrics $(g_t)_{t\in [0,T]}$, $T\in (0,\infty)$, which is $C^1$ in $t$.
For simplicity, we take the notation: for $X, Y\in TM$,
\begin{align*}
    &\mathcal{R}_t(X,Y):={\rm Ric}_t(X,Y)+\partial_tg_t(X,Y),
\end{align*}
 where ${\rm Ric}_t$ is the Ricci curvature tensor with respect to the metric $g_t$.
Suppose a smooth positive function $u:M\times [0,T]\rightarrow \mathbb{R}$ satisfies the heat equation
\begin{align}\label{heat-equ}
\frac{\partial}{\partial t}u(x,\cdot)(t)=\frac{1}{2}\Delta u(\cdot,t)(x)
\end{align}
on $M\times [0,T]$. This paper is devoted to the Hamilton type and Li-Yau type gradient estimates for  $u$ by
constructing  some suitable semimartingales.  To explain  the main idea, we take the Li-Yau type gradient estimate on the compact manifold carrying the Ricci flow for example. 
  Here and in what follows, the Ricci flow will mean (probabilistic convention):
\begin{align}\label{Ricci-flow}
\frac{\partial}{\partial t}g(x,t)=-{\rm Ric}(x,t), \ \  (x,t)\in M\times [0,T].
\end{align}
In this case, ${\mathcal R}_t=0$.  Moreover, suppose that there exists some constant $k>0$ such that
\begin{align*}
 0\leq {\rm Ric}_t\leq k
\end{align*}
holds on $M\times [0,T]$.
 Let $X_t^T$ be a $g_{(T-t)}$-Brownian motion  (see \cite{ACT} for the construction), which is ensured to be non-explosive under the Ricci flow  (see \cite{Ku}).   Let
\begin{align}\label{s-t}
\hat{S}_{t}&=h_t\l(\frac{|\nabla  u|^2}{u}(X^T_{T-t},t)-\Delta u_t\r)-n u_t\dot{h}_t,
\ \alpha>1,
\end{align}
where $u_t:=u(X^T_{T-t},t)$  and $h_t$ is the solution of $$\dot{h}_t=h_t(c_1t^{-1}+c_2k),\ \ h_T=1$$
for some positive constant $c_1, c_2$.
Then,
\begin{align}\label{mart-1}
\hat{S}_{t}=h_t\l\{\frac{|\nabla  u|^2}{u}(X^T_{T-t},t)- \Delta u(\cdot, t)(X_{T-t}^T)-nu_t(c_1t^{-1}+c_2k)\r\}.
\end{align}
Now,  the Li-Yau type gradient estimate can be derived  if we can choose some suitable constants $c_1, c_2$ in \eqref{mart-1} such that
$\hat{S}_{t}$ is a supermartingale,  see the proof of Lemma \ref{local-li-Yau} below.
   When it reduces to the constant metric case, the classical  Li-Yau inequality says that
$\frac{|\nabla u|^2}{u^2}-\frac{\Delta u}{u}$ can be dominated by $\frac{n}{t}$ and it does not need ${\rm Ric}_t\leq k$. However, in our setting, when using the It\^{o} formula, we need  this curvature condition to deal with some additional terms from the time derivative  of the metric, e.g. the derivative of  $\Delta u$ about $t$. It is  the main difficult
for us to overcome.



 The rest parts of the paper are organized as follows. In Sections 2,
we  study the Hamilton type  gradient estimates. In Section 3, we first prove  the Li-Yau type gradient inequality,
 and then discuss the local Li-Yau type gradient inequality with lower order term. In Section 4, we give
 the application to the Ricci flow.

 For readers' convenience,   we will take the same notations as in \cite{Sun11}. We emphasize that the Laplacian $\Delta$,
 the gradient $\nabla$, and the norm
$|\cdot|$ appearing above depend on the parameter $t\in [0,T]$.  The inner product $\l<\cdot,\cdot\r>$, the normal vector field $\frac{\partial}{\partial \nu}$ will be used in the following content also depend on $t$.

\section{Gradient estimates of Hamilton type}
\hspace{0.5cm}In this section, we explain how submartingales, related to positive solutions to the heat equation, can imply the Hamilton type
gradient estimates, i.e. the space-only gradient estimates.
For any subset $D\subset M\times [0,T]$ and $f$ defined on $M\times [0,T]$, $\|f\|_{D}:=\sup_{(x,t)\in D}|f|$.
The Hamilton type inequality on compact manifolds is presented as follows.
 \begin{theorem}[Gradient inequality of Hamilton type]\label{Hamilton-1}
Let $M$ be a compact manifold such that ${\mathcal R}_t \geq -k(t)$
for some $k\in C([0,T])$. Suppose that $u$ is a positive solution  to the
heat equation \eqref{heat-equ} on $M\times [0,T]$.  Then for $t\in (0,T]$, we have
\begin{align*}
   \frac{|\nabla u|^2}{u^2}(x,t)\leq\frac{ 2}{\int_{0}^te^{-\int_{s}^tk(r)\vd r}\vd s}\log \frac{\|u\|_{M\times [0,T]}}{u}.
\end{align*}
\end{theorem}

 The following two lemmas is essential to the proof of Theorem \ref{Hamilton-1}.  First,
  let us introduce basic formulas for solutions to the heat equation.
  \begin{lemma}\label{basic-formula}
Let $u=u(x,t)$ be a positive solution to \eqref{heat-equ}
on $M\times [0,T]$. Then the following equations hold:
\begin{align}
& \l(\frac{1}{2}\Delta -\partial_t\r)(u\log u)=\frac{1}{2}\frac{|\nabla u|^2}{u};\nonumber\\
& \l(\frac{1}{2}\Delta -\partial_t\r)\frac{|\nabla  u|^2}{u}=\frac{1}{u}\l|\nabla^2u-\frac{\nabla u\otimes \nabla  u}{u}\r|^2+\frac{{\mathcal{R}}_t(\nabla  u,\nabla  u)}{u}.\label{eq3}
\end{align}
\end{lemma}
The two equalities in Lemma \ref{basic-formula}, which can be checked directly, imply some inequalities frequently used in the sequel and crucial for our approach. Let $f=\log u$. We  see that
\begin{equation}\label{eq4}
\frac{1}{u}\l|\nabla^2 u-\frac{\nabla u\otimes \nabla  u}{u}\r|^2=u\l|\nabla^2 f\r|^2\geq \frac{u}{n}(\Delta f)^2=\frac{1}{nu}\l(\Delta u-\frac{|\nabla u|^2}{u}\r)^2.
\end{equation}
Then, if $\mathcal{R}_t\geq -k(t)$ on $M$ for some $k\in C([0,T])$,  we have
      \begin{align}
         \l(\frac{1}{2}\Delta -\partial_t\r)\frac{|\nabla  u|^2}{u}&\geq \frac{1}{nu}\l(\Delta  u-\frac{|\nabla  u|^2}{u}\r)^2+\frac{{\mathcal{R}}_t(\nabla u,\nabla u)}{u}\nonumber\\
        & \geq  -k(t)\frac{|\nabla  u|^2}{u}.\label{basic-ineq}
      \end{align}

Next, we need to construct a (sub)martingale as in \cite[Lemma 2.4]{AT}.
Recall that $(X_t^T)$ is a $g_{(T-t)}$-Brownian motion on $M$. Let $\{P_{s,t}\}_{0\leq s\leq t\leq T}$ be the associated semigroup. It is easy to see that $P_{T-t,T}f$ is a solution to the equation \eqref{heat-equ}.
\begin{lemma}\label{lem2}
Let
$u(x,t)=P_{T-t,T}f(x)$ be a positive solution to the heat equation \eqref{heat-equ} on $M\times [0,T]$. If
$\mathcal{R}_t\geq -k(t)$ for some $k\in C([0,T])$, then for any $g_{(T-t)}$-Brownian motion $(X_t^T)$ on $M$, the process
\begin{align}\label{Nt}
    H_t:=h(t)\frac{|\nabla P_{T-t,T}f|^2}{P_{T-t,T}f}(X^T_{T-t})+(P_{T-t,T}f\log P_{T-t,T}f)(X_{T-t}^T)
\end{align}
with $h(t)={\frac{1}{2}\int_{0}^te^{-\int_{s}^tk(r)\vd r}\vd s}$
is a local supermartingale $($up to lifetime$)$.
\end{lemma}
\begin{proof}
It can be checked directly by the It\^{o} formula,  we omit the details.
\end{proof}
\begin{proof}[{\bf Proof of Theorem \ref{Hamilton-1} }]
Since $M$ is compact, by Lemma \ref{lem2}, the local submartingale $H_t$ defined in \eqref{Nt} is a true supermartingale. Then,  we have, for $t\in (0,T]$,
$$\mathbb{E}(H_t|X_{T-t}^T=x)\leq \mathbb{E}(H_0|X_{T-t}^T=x).$$
That is
$$\l|\frac{\nabla  u}{u}\r|^2(x,t)\leq \frac{ 2}{\int_{0}^te^{-\int_{s}^tk(r)\vd r}\vd s}P_{T-t,T}\l(\frac{f}{P_{T-t,T}f}\log \frac{f}{P_{T-t,T}f}\r)(x).$$
By normalizing $f$ as $f^*=f/P_{T-t,T}f$, we complete the proof.
\end{proof}

Let $\rho_t(x,y)$ be the distance between $x\in M$ and $y\in M$ with respect to the metric $g_t$. Fix $x_0\in M$ and $\rho>0$. The notation $B_{\rho,T}$ stands for the set $\{(x,t)\in M\times [0,T]\ |\ \rho_t(x,x_0)<\rho\}$.
Our next step is to localize the arguments to cover the solution to the heat equation on $B_{\rho,T}$.
\begin{theorem}[Local gradient inequality of Hamilton type]\label{cor1}
Assume that there exist some nonnegative constants $k_1,k_2$ such that
 \begin{align}\label{condition-1}
 {\rm Ric}_t\geq -k_1,\ \ \  \partial_tg_t\geq - k_2
 \end{align}
  holds   on $B_{\rho,T}$. Let  $u$ be a solution to the heat equation on $B_{\rho,T}$, which is positive and
continuous on $\overline{B_{\rho,T}}$. Then, for each $(x,t)\in B_{\rho/2,T}$ and $ t\neq 0$, there holds
\begin{align*}
    \l|\frac{\nabla u}{u}\r|^2(x,t)\leq 2\l[\frac{1}{t}+\frac{4\pi^2(n+7)}{(4-\pi)^2\rho^2}+\frac{(\pi^2+16)(k_1+k_2)}{(4-\pi)^2}\r]\l(4+\log \frac{\|u\|_{\overline{B_{\rho,T}}}}{u}\r)^2.
\end{align*}
\end{theorem}

 We want to introduce the Hamilton type inequity  on any relatively compact subset $D$ with nonempty smooth boundary first.
\begin{lemma}\label{local-gradient-estimate}
Let $D\subset M\times [0,T]$ be a relatively compact subset  with nonempty smooth boundary. Assume
 \eqref{condition-1}
 holds for some positive constants $k_1,k_2$ on $D$. Let  $u$ be a solution to the heat equation on $D$, which is positive and
continuous on $\overline{D}$. Let $\varphi\in C^{1,2}(\bar{D})$ with $\varphi>0$ and $\varphi|_{\partial D}=0$. Then, for $(x,t)\in D$ and $ t\neq 0$, there holds
\begin{align}\label{ineq-Li-Yau}
    \l|\frac{\nabla u}{u}\r|^2(x,t)\leq 2\l(\frac{1}{t}+\frac{\sup_{D}\{7|\nabla \varphi_t|^2-\varphi_t(\Delta -2\partial_t)\varphi_t\}}{\varphi_{t}^2(x)}+k_1+k_2\r)\l(4+\log \frac{\|u\|_{\overline{D}}}{u}\r)^2.
\end{align}
\end{lemma}
To simplify the proof, we  set
$$u_{t}=u(X_{T-t}^T, t), \, \nabla u_t=\nabla u(\cdot, t)(X_{T-t}^T)\ \mbox{and}\  \Delta u_{t}=\Delta u(\cdot, t)(X_{T-t}^T).$$
  Let also
  \begin{align}\label{Q-mart}
  q(x,t)=\frac{|\nabla u|^2}{u}(x,t)\ \ \mbox{and}\ \ q_t=q(X_{T-t}^T,t).
  \end{align}
\begin{proof}[{\bf Proof\ of\ Lemma \ \ref{local-gradient-estimate}}]
The proof is  due to  that of \cite[Theorem 6.1]{AT}.
 We now study the following process on $D$,
$$S_t=\frac{|\nabla  u_t|^2}{2u_t}-u_t(1-\log u_t)^2Z_t,$$
where $$ Z_t:=\frac{c_1}{t}+\frac{c_2}{\varphi^2_t(X^T_{T-t})}+c_3 $$ for some constants $c_1,c_2,c_3>0$, which will be specified later.
We now  turn to investigate the martingale property of the process $S_t$.
 It is easy to see that the process $(X_{T-t}^T)$ is generated by
$-\frac{1}{2}\Delta $ and
it solves the
equation
$$\vd X^T_{T-t}=- U^T_{T-t}\circ \vd B_t,\  X^T_0:=\mathbf{p}U^T_0=x\in D,$$
up to the life time, where $U^T_{T-t}$ is the horizontal process of $ X^T_{T-t}$.
 Assume that $0<u\leq e^{-3}$ (the assumption will be removed in Lemma \ref{local-gradient-estimate} through replacing $u$ by $e^{-3}u/\|u\|_{D}$).
Using the fact that
 $$\nabla (u_t(1-\log u_t)^2)= (\log ^2u_t-1)\nabla u_t\ \  \mbox{and}\ \ \vd (u_t(1-\log u_t)^2){=} -q_t\log u_t\vd t, $$
 we get (the  indicator ``$\ \stackrel{m}{}\ $" stands for the modula differential of local martingales)
 \begin{align*}
 \vd S_t= &\frac{1}{2}\vd \l\{\frac{|\nabla u_t|^2}{u_t}\r\}-u_t(1-\log u_t)^2\vd Z_t-Z_t\vd (u_t(1-\log u_t)^2)\\
 & +2c_2(1-\log u_t)(1+\log u_t)\varphi_t^{-3}(X_{T-t}^T)\l<\nabla u_t,\nabla \varphi_t(X^T_{T-t})\r>\vd t\\
\stackrel{m}{\leq} & \frac{1}{2}(k_1+k_2)q_t\vd t+u_t(1-\log u_t)^2\l[c_1t^{-2}+c_2c_{\varphi}(X_{T-t}^T ,t)\varphi_t^{-4}(X_{T-t}^T)\r]\vd t+\log u_t q_t Z_t\vd t\\
 &+2c_2(1-\log u_t)(1+\log u_t)\varphi_t^{-3}(X^T_{T-t})\l<\nabla u_t,\nabla \varphi_t(X_{T-t}^T)\r>\vd t,
 \end{align*}
 where $c_{\varphi}(x,t)=\l[3|\nabla \varphi_t|^2-\varphi_t(\Delta -2\partial_t)\varphi_t\r](x).$
 More, we have $2\log u_t\leq -3(1-\log u_t)/2$ from $u\leq e^{-3}$, which, together with $|1+\log u_t|\leq 1-\log u_t$, yields
 \begin{align*}
 \vd S_t\stackrel{m}{\leq} & (\log u_t-1)\bigg\{\frac{1}{2}\l(\frac{3}{2}Z_t-k_1-k_2\r)q_t-u_t(1-\log u_t)\l[c_1t^{-2}-c_2c_{\varphi}(X_{T-t}^T,t)\varphi_t^{-4}(X_{T-t}^T)\r]\\
&-c_22\sqrt{2} (1-\log u_t)\varphi_t^{-2}(X_{T-t}^T)|\nabla \varphi_t(X_{T-t}^T)|\sqrt{u_t}\times\varphi_t^{-1}(X_{T-t}^T)\sqrt{\frac{q_t}{2}}\bigg\}\vd t\\
\leq & (\log u_t-1)\bigg\{\frac{1}{2}\l(\frac{3}{2}Z_t-k_1-k_2\r)q_t-u_t(1-\log u_t)\l[c_1t^{-2}-c_2c_{\varphi}(X_{T-t}^T,t)\varphi_t^{-4}(X_{T-t}^T)\r]\\
&-4c_2(1-\log u_t)^2\varphi_t^{-4}(X_{T-t}^T)|\nabla \varphi_t(X_{T-t}^T)|^2u_t-c_2 \varphi_t^{-2}(X_{T-t}^T)\frac{q_t}{4}\bigg\}\vd t\\
\leq & (\log u_t-1)\bigg\{
-u_t(1-\log u_t)^2 \times \l[c_1 t^{-2}+c_2\l(c_{\varphi}(X_{T-t}^T,t)+4|\nabla \varphi_t(X_{T-t}^T)|^2\r)\varphi_t^{-4}(X_{T-t}^T)\r]\\
&\qquad \qquad  \quad +(Z_t-k_1-k_2)\frac{1}{2}q_t
\bigg\}\vd t.
 \end{align*}
 Let $c_1=1, c_2=\sup_{D}\{c_{\varphi}+4|\nabla \varphi_t|^2\}$ and $c_3=k_1+k_2$. We conclude
 \begin{align*}
 \vd S_t&\stackrel{m}{\leq} -(1-\log u_t)(Z_t-k_1-k_2)\l[\frac{1}{2}q_t-u_t(1-\log u_t)^2Z_t\r]\vd t\\
 &=-(1-
 \log u_t)(Z_t-k_1-k_2)S_t\vd t.
 \end{align*}
 Now we consider the process $\{X^T_{T-s}\}_{s\in [0,t]}$ starting from $x$ at time $t$.
 This process has nonpositive drift on $\{S_s\geq 0\}$. On the other hand, $S_s$ converges to $-\infty$ as $s\rightarrow 0\vee \tau(x)$, where $\tau(x):=\sup\{s\in[0,t]: X_{T-s}^T\notin D, X_{T-t}^T=x\},\ \sup \varnothing =0$.
 Therefore, $S_t\leq 0$, i.e.
 $$\frac{|\nabla u_t|^2}{2u_t}\leq (1-\log u_t)^2 Z_t.$$
 Replacing $u$ by $e^{-3}u/\|u\|_D$, we complete the proof.
\end{proof}

 \begin{proof}[{\bf Proof of Theorem \ref{cor1}}]
   To specify the constants of \eqref{ineq-Li-Yau} in terms of $\varphi$,  an explicit choice for $\varphi$ has to be done.   For any $y\in B_{\rho,T}$, i.e.  $\rho_t(x_0,y)\leq \rho$, $t\in [0,T]$,
 \begin{align}\label{varphi}
  \varphi(t,y):=\cos\frac{\pi \rho_{t}(x_0,y)}{2\rho}.
 \end{align}
 Then, $\varphi$ is nonnegative and bounded by $1$, and $\varphi$ vanishes on $\partial B_{\rho,T}$.
 Fix $(x,t)\in B_{\rho/2,T}$, we have
 \begin{equation}\label{estimate-constant-1}
 \varphi_{t}(x)= \varphi(t,x)=\cos\frac{\pi \rho_{t}(x_0,x)}{2\rho}\geq 1-\frac{\pi \rho_{t}(x_0,x)}{2\rho}\geq 1-\frac{\pi}{4}.
\end{equation}
By convention, $\Delta \varphi_t^{-2}=0$, when $\varphi_t^{-2}$ is not differentiable.  As a consequence, all the estimates in Lemma \ref{local-gradient-estimate}  still hold true with $\varphi$ defined by \eqref{varphi}.

 Now, we derive  explicit estimates for the constants. To this end, we observe that
 \begin{equation}\label{estimate-constant-3}
 \|\nabla \varphi_t\|_{B_{\rho,T}}\leq \frac{\pi}{2\rho}.
\end{equation}
Since ${\rm Ric}_t \geq -k_1$, $\partial_tg_t \geq -k_2$ for $k_1,k_2>0$, by the index lemma, we have
 \begin{align*}
    -(\Delta -2\partial_t)\varphi_t
    = &\sin \frac{\pi\rho_{t}(x_0,y)}{2\rho}\frac{\pi}{2\rho}(\Delta -2\partial_t)\rho_{t}(x_0,\cdot)(y)+\cos\frac{\pi \rho_t(x_0,y)}{2\rho}\cdot \frac{\pi^2}{4\rho^2}|\nabla \rho_{t}|^2\nonumber\\
    \leq & \sin \frac{\pi \rho_{t}(x_0,y)}{2\rho}\frac{\pi}{2\rho}\l[\frac{n-1}{\rho_{t}(x_0,y)}+(k_1+k_2)\rho_{t}(x_0,y)\r]+\frac{\pi^2}{4\rho^2}\nonumber\\
    \leq &\frac{\pi^2n}{4\rho^2}+\frac{\pi}{2}(k_1+k_2).
 \end{align*}
Thus, we further have
\begin{equation*}
\sup_{B_{\rho,T}}\{7|\nabla \varphi_t|^2-\varphi_t(\Delta -2\partial_t)\varphi_t\}\leq \frac{\pi^2(n+7)}{4\rho^2}+\frac{\pi}{2}(k_1+k_2).
\end{equation*}
Combining this with  \eqref{ineq-Li-Yau} and \eqref{estimate-constant-1}, we complete the proof.
\end{proof}

\section{ Li-Yau type gradient inequalities}
\hspace{0.5cm} This section is devoted to extending the argument of the Li-Yau type inequalities on Riemannian manifolds to our setting.
\subsection{Li-Yau type  inequalities}
\hspace{0.5cm}
First, we present the local version of the Li-Yau gradient inequality as follows.
 \begin{theorem}[Local gradient inequality of Li-Yau type]\label{local-th}
Assume that there exist some nonnegative constants $k_1,k_2,k_3,k_4$ such that
 \begin{align}\label{condition-2}
 {\rm Ric}_t\geq -k_1,\ \ \ -k_2 \leq \partial_tg_t\leq k_3,\ \ \ |\nabla (\partial_tg_t)|\leq k_4
 \end{align}
  holds   on $B_{\rho,T}$. Let  $u$ be a solution to the heat equation on $B_{\rho, T}$, which is positive and
continuous on  $\overline{B_{\rho, T}}$.  For any $\alpha>1$ and $t>0$,
\begin{align*}
\frac{|\nabla u|^2}{u^2}-\alpha\frac{\Delta u}{u}\leq  n\alpha^2 \bigg[&\frac{2}{t}+\frac{8\pi^2}{(4-\pi)^2\rho^2}\l(n+3+\frac{\alpha^2 n}{\alpha-1}\r) +\frac{16{\pi}(k_1+k_2)}{(4-\pi)^2}
    +\max\{k_2,k_3\}+k_3\nonumber\\
    &+\sqrt{{2 k_4}}+\frac{k_1+ k_4}{\alpha-1}\bigg]
    \end{align*}
    holds on $B_{\rho/2,T}$.
\end{theorem}

Parallel to Theorem \ref{cor1}, we need to consider the local Li-Yau inequality on some subset $D$.
\begin{lemma}\label{Li-Yau-gradient-inequality}
Let $D\subset M\times [0,T]$ be a relatively compact subset  with nonempty smooth boundary. Assume
 \eqref{condition-2}
 holds for some non-negative constants $k_1,k_2,k_3,k_4$ on $D$. Let  $u$ be a solution to the heat equation on $D$, which is positive and
continuous on $\overline{D}$. Let $\varphi\in C^{1,2}(\overline{D})$ with $\varphi>0$ and $\varphi|_{\partial D}=0$.   For any $\alpha>1$ and $t>0$, we have
\begin{align*}
\frac{|\nabla u|^2}{u^2}-\alpha\frac{\Delta u}{u}\leq  n\alpha^2 \bigg(&\frac{2}{t}+\frac{2C_{\varphi,\alpha,n}}{\varphi_{t}^{2}(x)}
    +\max\{k_2,k_3\}+k_3+\sqrt{{2 k_4}}+\frac{k_1+ k_4}{\alpha-1}\bigg)
    \end{align*}
    holds on $D$,
where $C_{\varphi,\alpha,n}:=\sup_D\{(3+\alpha^2(\alpha-1)^{-1} n)|\nabla \varphi_t|^2-\varphi_t(\Delta -2\partial_t)\varphi_t\}$.
\end{lemma}
\begin{proof}[{\bf Proof}]
Define
\begin{align*}
   Y_t=c_1t^{-1}+c_2\varphi_t^{-2}(X_{T-t}^T)+c_3\max\{k_2,k_3\}+c_4\sqrt{k_4}+\frac{k_1+(\alpha-1) k_3+ k_4}{\alpha-1},\ \alpha>1,
\end{align*}
where $c_1,c_2, c_3, c_4>0$ are constants, which will be specified later.
Let $h_t$ be the solution of  $\dot{h}_t=h_tY_t,\ h_T=1$.
Consider  the following process,
\begin{align*}
S_{t,\alpha}:=h_t(q_t-\alpha\Delta u_t)-\beta nu_t\dot{h}_t\equiv h_t\l(q_t-\alpha\Delta u_t-n\beta u_t Y_t\r),
\end{align*}
where  $\beta>0$ will also be specified later and $q_t$ is defined in \eqref{Q-mart}.
We first investigate the martingale property of $S_{t,\alpha}$.
By \eqref{eq3} and \eqref{eq4},
we find
$$\vd (h_tq_t)\stackrel{m}{\leq}  \bigg[-h_tu_t|\nabla^2f|^2-\dot{h}_tq_t+h_tu_t{\rm Ric}_t(\nabla f,\nabla f)+h_tu_t\partial_tg_t(\nabla f,\nabla f)\bigg]\vd t, $$
By this and observing from \cite[Lemma 3.2]{Sun11},
\begin{align*}
\vd (\Delta u_t)&\stackrel{m}{=}\l[-u_t\l<\partial_tg_t,\nabla^2f\r>-u_t \partial_tg_t(\nabla f,\nabla f )-u_t\l<{\rm div}(\partial_tg_t)-\frac{1}{2}\nabla ({\rm tr}_g (\partial_tg_t)), \nabla f\r>\r](X_{T-t}^T,t)\vd t\\
 &\geq \l[-u_t\l<\partial_tg_t,\nabla^2 f\r>-u_t \partial_tg_t(\nabla f,\nabla f )-2\sqrt{n}k_4 u_t |\nabla f|\r](X_{T-t}^T,t)\vd t\\
 & \geq \l[-u_t\l<\partial_tg_t,\nabla^2 f\r>-u_t \partial_tg_t(\nabla f,\nabla f )-\alpha n k_4u_t|\nabla f|^2-\alpha^{-1} k_4 u_t \r](X_{T-t}^T,t)\vd t,
\end{align*}
 we have
\begin{align*}
\vd S_{t,\alpha}{=}&\vd h_tq_t-n\beta u_t\vd (h_tY_t)-n\beta\vd [u,\dot{h}]_t-\alpha\vd h_t\Delta u_t\nonumber\\
\stackrel{m}{\leq} & \big[-h_tu_t|\nabla^2f|^2+\alpha h_tu_t\l< \partial_tg_t, \nabla^2f\r>-\alpha\dot{h}_t(\Delta u_t)\big]\vd t\nonumber\\
&+(\dot{h}_t+k_1h_t+ k_4h_t+(\alpha-1)k_3h_t)q_t \vd t-n\beta u_th_tY_t^2\vd t\nonumber\\
&+n\beta u_t h_t\big[c_1t^{-2}+c_2c_{\varphi}(X_{T-t}^T,t)\varphi^{-4}_t(X_{T-t}^T)+\beta^{-1}k_4\alpha^2\big]\vd t-n c_2\beta h_t\vd [u,\varphi_{\cdot}^{-2}(X^T_{T-\cdot})]_t,
\end{align*}
where $c_{\varphi}(x,t)=\l[3|\nabla \varphi_t|^2-\varphi_t(\Delta -2\partial_t)\varphi_t\r](x).$
As $\Delta f_t=\frac{1}{u_t}(\Delta u_t-q_t)$, we have
\begin{align}\label{estimate-S1-2}
\vd S_{t,\alpha}\stackrel{m}{\leq} &\big[-h_tu_t|\nabla^2f|^2(X_{T-t}^T,t)+\alpha h_tu_t\l<\partial_tg_t,\nabla^2f\r>(X_{T-t}^T,t)-\alpha h_tY_tu_t(\Delta f_t) \big]\vd t\nonumber\\
&+(k_1+(\alpha-1) k_3+ k_4)h_tq_t \vd t-(\alpha-1)\dot{h}_tq_t\vd t
-n\beta u_th_tY_t^2\vd t\nonumber\\
&+n\beta u_t h_t\big[c_1t^{-2}
+c_2c_{\varphi}(X_{T-t}^T, t)\varphi^{-4}_t(X_{T-t}^T)+\beta^{-1} k_4\alpha^2\big]\vd t\nonumber\\
&-n c_2\beta h_t\vd [u,\varphi_{\cdot}^{-2}(X^T_{T-\cdot})]_t.
\end{align}
For any $a,b>0$ such that $a+b=\alpha^{-1}$,
\begin{align}\label{esti-hess}
    |\nabla^2f|^2-\alpha\l<\partial_tg_t,\nabla^2f\r>
=&\l\{\l(a\alpha+b\alpha\r)|\nabla^2f|^2-\alpha\r<\partial_tg_t,\nabla^2f\l>\r\}\nonumber\\
=&\bigg[a\alpha|\nabla^2f|^2+\alpha\l|\sqrt{b}\nabla ^2 f-\frac{1}{2\sqrt{b}}\partial_tg_{t}\r|^2-\frac{\alpha}{4b}|\partial_tg_{t}|^2\bigg]\nonumber\\
\geq &\bigg[a\alpha|\nabla^2f|^2-\frac{\alpha}{4b}|\partial_tg_{t}|^2\bigg]
\geq \frac{a\alpha}{n}(\Delta f)^2-\frac{\alpha n}{4b}\max \{k_2^2,k_3^2\}
\end{align}
at $(x,t)\in D$. So,
\begin{align}\label{s2-1}
|\nabla^2f|^2-\alpha\l<\partial_tg_t,\nabla^2f\r> -\alpha Y_t (\Delta f) &
\geq \frac{a\alpha}{n} (\Delta f)^2-\frac{\alpha b}{4b}\max\l\{k_2^2,k_3^2\r\}-\alpha (\Delta f)Y_t\nonumber\\
&\geq -\frac{n\alpha}{4a}Y_t^2-\frac{\alpha n}{4b} \max\{k_2^2,k_3^2\}.
\end{align}
In addition, for the last term in \eqref{estimate-S1-2}, we have
\begin{align}\label{estimate-S2}
    -n\beta & c_2h_t \vd [u,\varphi_{\cdot}^{-2}(X_{T-\cdot}^T)]_t=2n\beta c_2h_t\varphi_t^{-3}\l<\nabla u_t,\nabla \varphi_t(X^T_{T-t})\r>\vd t\nonumber\\
    &\leq 2n c_2h_t\l[\varphi^{-1}_t((\alpha-1)^{-1}nu_t)^{-1/2}|\nabla u_t|\times \beta((\alpha-1)^{-1}nu_t)^{1/2}\varphi_t^{-2}|\nabla \varphi_t|\r](X^T_{T-t})\vd t\nonumber\\
    & \leq \l[(\alpha-1)c_2h_t\varphi_t^{-2}u_t^{-1}|\nabla u_t|^2+n^2c_2h_t\beta^2(\alpha-1)^{-1}u_t\varphi_t^{-4}
   |\nabla \varphi_t|^2\r](X^T_{T-t})\vd t\nonumber\\
    &\leq \l\{[(\alpha-1)\dot{h}_t-(k_1+k_4+(\alpha-1) k_3)h_t]q_t+c_2(\alpha-1)^{-1} \beta^2 n^2u_th_t\varphi_t^{-4}|\nabla \varphi_t|^2(X^T_{T-t})\r\}\vd t.
\end{align}
 Combining this with  \eqref{estimate-S1-2}, \eqref{s2-1} and \eqref{estimate-S2}, we arrive at
\begin{align*}
\vd S_{t,\alpha}\stackrel{m}{\leq} & \l[-nh_tY_t^2u_t\l(\beta-\frac{\alpha}{4a}\r)+\frac{n\alpha}{4b}u_th_t\max\{k_2^2,k_3^2\}\r]\vd t\nonumber\\
&+nu_th_t\l(\beta c_1t^{-2}+\beta c_2C_{\varphi,\alpha,\beta,n}\varphi_t^{-4}(X_{T-t}^T)+\alpha^2 k_4\r)\vd t,
\end{align*}
where $a+b=\alpha^{-1}$ and $C_{\varphi,\alpha,\beta,n}=\sup_D\{(3+\beta(\alpha-1)^{-1} n)|\nabla \varphi_t|^2-\varphi_t(\Delta -2\partial_t)\varphi_t\}$.
Let
\begin{align}\label{restrict}
    &\l(\beta-\frac{\alpha}{4a}\r)c_1^2\geq \beta c_1, \ \ \ \
     \l(\beta-\frac{\alpha}{4a}\r)c_2^2-\beta c_2C_{\varphi, \alpha, \beta, n}\geq 0,\ \  \nonumber\\
     & \l(\beta-\frac{\alpha}{4a}\r)c_3^2-\frac{\alpha}{4b}\geq 0,\ \ \l(\beta-\frac{\alpha}{4a}\r)c_4^2-\alpha^2\geq 0.
\end{align}
Using $$Y_t^2\geq c_1^2t^{-2}+c_2^2\varphi_t^{-4}(X_{T-t}^T)+c_3^2\max\{k_2^2,k_3^2\}+c_4^2k_4,$$
we obtain that the right side of \eqref{estimate-S1-3} is non-positive.
Consider the process $\{X_{T-s}^T\}_{s\in [0,t]}$ with $X_{T-t}^T=x$ and $(x,t)\in D$. Let $\tau(x):=\sup\{ s<t:X_{T-s}^T\notin D, s\geq 0\}$ and $\sup\varnothing =0$.
Since $q_t-\alpha\Delta u_t-nu_tY_t$ converges to $-\infty$ as $t$ tends to $0\vee\tau(x)$, the process $S_{t,\alpha}$ has non-positive drift, i.e.  $S_{t,\alpha}(x)\leq 0$, which implies
\begin{align}\label{estimate-S2-1}
    \frac{|\nabla u|^2}{u^2}-\alpha\frac{\Delta u}{u}\leq n\beta \l(\frac{c_1}{t}+\frac{c_2}{\varphi_{t}^{2}(x)}+c_3\max\{k_2,k_3\}+c_4\sqrt{k_4}+\frac{k_1+ k_4}{\alpha-1}+k_3\r).
\end{align}

Set $a=b=\frac{1}{2\alpha}$ and  $\beta=\frac{\alpha}{2a}=\alpha^2>0$. Let $$C_{\varphi,\alpha,n}:=\sup_D\{(3+\alpha^2(\alpha-1)^{-1} n)|\nabla \varphi_t|^2-\varphi_t(\Delta +2\partial_t)\varphi_t\}.$$ Then, from \eqref{restrict},
we  select $c_1, c_2,c_3,c_4$ such that
$$ c_3= 1, c_2=2C_{\varphi,\alpha,n},  c_1= 2,\  c_4=\sqrt{{2}},$$
 we complete the proof by taking these constants $c_1,c_2,c_3, c_4$ into \eqref{estimate-S2-1}.
\end{proof}
\begin{proof}[{\bf Proof of Theorem \ref{local-th}}]When  $D=B_{\rho,T}$, the function $\varphi$ is chosen as  in \eqref{varphi}.
Moreover, let ${\rm Cut}_t(x_0)$ be the set of the $g_t$ cut-locus of $x_0$ on $M$. Since
 the time spent by $X_{T-t}^T$ on $\bigcup_{t\in [0,T]}{\rm Cut}_t(x_0)$  is a.s. zero (see \cite{Ku}), the differential of the brackets $[\varphi_{\cdot}(X_{T-\cdot}^T),u_{\cdot}]_t$ may be taken as 0 at points where $\varphi_t^{-2}$ is not differentiable. Therefore,
\begin{align*}
C_{\varphi,\alpha,n}&=\sup_{B_{\rho,T}}\{(3+\alpha^2(\alpha-1)^{-1} n)|\nabla \varphi_t|^2-\varphi_t(\Delta -2\partial_t)\varphi_t\}\nonumber\\
    &\leq \frac{\pi^2(n+3+\alpha^2(\alpha-1)^{-1} n)}{4\rho^2}+\frac{\pi}{2}(k_1+k_2),
    \end{align*}
combining this with Lemma \ref{Li-Yau-gradient-inequality}, we complete the proof.
\end{proof}
\begin{theorem}[Gradient inequality of Li-Yau type]\label{Li-Yau-gradient-inequality-1}
Let $M$ be a compact manifold.
Assume \eqref{condition-2} holds for some nonnegative constants $k_1,k_2,k_3,k_4$ on $M\times [0,T]$. Let  $u$ be the positive solution of the heat equation \eqref{heat-equ} on $M\times [0,T]$. For any $\alpha>1$, we have
\begin{align}\label{local-li-Yau}
\frac{|\nabla u|^2}{u^2}-\alpha\frac{\Delta u}{u}\leq  n\alpha^2 \bigg(&\frac{2}{t}+\max\{k_2,k_3\}+k_3+\sqrt{2 k_4}+\frac{k_1+ k_4}{\alpha-1}\bigg)
    \end{align}
    holds on $M\times [0,T]$.
\end{theorem}
\begin{proof}
Define $$\tilde{Y}_t=c_1t^{-1}+c_2\max\{k_2,k_3\}+c_3\sqrt{k_4}+k_3+\frac{k_1+k_4}{\alpha-1},$$
where $c_1,c_2,c_3$ are constants, which will be specified later. Let $h_t$ be the solution of $$\dot{h}_t=h_tY_t,\ \ h_T=1.$$
Consider the following process
$$\tilde{S}_{t,\alpha}:=h_t(q_t-\alpha\Delta u_t-n\beta u_t \tilde{Y}_t),$$
where $\beta>0$ will be also specified later. By a similar discussion as in the proof of Lemma \ref{Li-Yau-gradient-inequality}, we obtain
\begin{align*}
\vd \tilde{S}_{t,\alpha}\stackrel{m}{\leq}
 & \l[-nh_tY_t^2u_t\l(\beta-\frac{\alpha}{4a}\r)+\frac{n\alpha}{4b}u_th_t\max\{k_2^2,k_3^2\}\r]\vd t\nonumber\\
&+nu_th_t\l(\beta c_1t^{-2}+\alpha^2 k_4\r)\vd t+(k_1+(\alpha-1) k_3+ k_4)h_tq_t \vd t-(\alpha-1)\dot{h}_tq_t\vd t\nonumber\\
\leq & \l[-nh_tY_t^2u_t\l(\beta-\frac{\alpha}{4a}\r)+\frac{n\alpha}{4b}u_th_t\max\{k_2^2,k_3^2\}\r]\vd t+nu_th_t\l(\beta c_1t^{-2}+\alpha^2 k_4\r)\vd t.
\end{align*}
The  rest part is similar to that of Lemma \ref{Li-Yau-gradient-inequality}, we omit it here.
\end{proof}
\subsection{Li-Yau type inequality with lower order term}
\hspace{0.5cm} Now, we  turn to studying the Li-Yau type inequality with the lower order term. The main result of this subsection is the following.
\begin{theorem}\label{cor2}
   Under the same condition as in Theorem \ref{cor1}. Then, for $t>0$,
\begin{align*}
\frac{|\nabla u|^2}{u^2}-\frac{\Delta u}{u}\leq & \frac{2n}{t}+\frac{8n\pi^2(n+3)}{(4-\pi)^2\rho^2}+\frac{16n\pi(k_1+k_2)}{(4-\pi)^2}+\max\{k_2,k_3\}n+\sqrt{2k_4}n
\\
&+ \bigg[\frac{8\pi n}{(4-\pi)^2\rho}+\sqrt{2n(k_1+k_4)}\bigg]\l\|\frac{|\nabla^{\cdot}u|}{u}\r\|_{B_{\rho,T}}\end{align*}
holds on  $ B_{\rho/2,T}$.
\end{theorem}
To prove this theorem, we first give a more general result.
\begin{lemma}[Local Li-Yau type with lower order term]\label{Li-Yau-lower-order}
 Under the same condition as in Lemma \ref{local-gradient-estimate}.  Then,
 for  $t>0$,
\begin{align*}
\frac{|\nabla u|^2}{u^2}-\frac{\Delta u}{u}\leq & \frac{2n}{t}+\frac{2n\sup_{D}\{3|\nabla \varphi_t|^2-\varphi_t(\Delta -2\partial_t)\varphi_t\}}{\varphi_{t}^2(x)}+\max\{k_2,k_3\}n+\sqrt{2k_4}n
\nonumber \\
&+\bigg(4n\frac{\|\varphi_{\cdot}\nabla^{\cdot}\varphi_{\cdot}\|_{D}}{\varphi_{t}^2(x)}+\sqrt{2n(k_1+k_4)}\bigg)\l\|\frac{|\nabla^{\cdot}u|}{u}\r\|_{D}
\end{align*}
holds on $D$.
\end{lemma}
\begin{proof}
Let $$\tilde{S}_t=\frac{|\nabla u_t|^2}{u_t}-\Delta u_t-nu_t \tilde{Z}_t,$$ where $\tilde{Z}_t=c_1t^{-1}+c_2\varphi_t^{-2}(X_{T-t}^T)+c_3$ with constants $c_1,c_2,c_3>0$ to be specified later. Let
 $$c_{\varphi}:=\sup_{D}\{3|\nabla \varphi_t|^2-\varphi_t(\Delta -2\partial_t)\varphi_t\}.$$
 By a similar calculation as in the proof of Lemma \ref{Li-Yau-gradient-inequality},
 we have
 \begin{align*}
    \vd \tilde{S}_t\stackrel{m}{\leq} &\l[-\frac{a}{nu_t}\l(\frac{|\nabla u_t|^2}{u_t}-\Delta u_t\r)^2+(k_1+k_4)\frac{|\nabla u_t|^2}{u_t}+\frac{n}{4b}u_t \max\{k_2^2,k_3^2\}+nu_tk_4\r]\vd t\\
    &+nu_t (c_1t^{-2}+c_2c_{\varphi}\varphi_t^{-4}(X_{T-t}^T))\vd t-n c_2 \vd [\varphi_{\cdot}^{-2}(X_{T-{\cdot}}^T), u]_t,
 \end{align*}
 where $a+b=1$.
 Since
 \begin{align*}
    -n c_2\vd [\varphi_t^{-2}(X_{T-t}^T), u]_t=&2n  c_2\varphi_t^{-4}(X_{T-t}^{T})\l<\frac{\nabla  u_t}{u_t}, \varphi_t\nabla  \varphi_t(X_{T-t}^T)\r>u_t\vd t\\
    \leq & 2nc_2u_t \l \|\frac{|\nabla^{\cdot}u|}{u}\r\|_{D}\|\varphi_{\cdot}\nabla^{\cdot}\varphi_{\cdot}\|_{D}\varphi_t^{-4}(X_{T-t}^T)\vd t,
 \end{align*}
 we arrive at
 \begin{align*}
 \vd \tilde{S}_t\stackrel{m}{\leq} & \bigg\{-\frac{a}{nu_t}\l(\frac{|\nabla u_t|^2}{u_t}-\Delta u_t\r)^2+nu_t\bigg[\frac{ c_1}{t^2}+\frac{ c_2}{\varphi_t^4(X_{T-t}^T)}\times\bigg(c_{\varphi}+2\l\|\frac{|\nabla^{\cdot}u|}{u}\r\|_{D}\|\varphi_{\cdot}\nabla^{\cdot}\varphi_{\cdot}\|_{D}\bigg)\bigg]\bigg\}\vd t\\
 &+nu_t\frac{k_1+k_4}{n}\l\|\frac{|\nabla^{\cdot}u|}{u}\r\|^2_{D}\vd t+\frac{1}{4b}nu_t\max\{k_2^2,k_3^2\}\vd t+nu_tk_4\vd t.
 \end{align*}
Then, letting $c_1=\frac{1}{a}$, $c_2=\frac{1}{a} \l[c_{\varphi}+2\l\|\frac{|\nabla^{\cdot}u|}{u}\r\|_{D}\|\varphi_{\cdot}\nabla^{(T-{\cdot})}\varphi_{\cdot}\|_{D}\r]$, and
$$c_3=\sqrt{\frac{k_1+k_4}{a n}}\l\|\frac{|\nabla^{\cdot} u|}{u}\r\|_{D}+\sqrt{\frac{1}{4ab}}\max\{k_2,k_3\}+\sqrt{\frac{k_4}{a}},$$
we obtain
$$ \vd \tilde{S}_t\leq-\frac{a\tilde{S}_t}{nu_t}\l(\frac{|\nabla u_t|^2}{u_t}-\Delta u_t+nu_tZ_t\r)\vd t.$$
Then, on $\{\tilde{S}_t\geq 0\}$, the process $\tilde{S}_t$ has nonpositive drift.
 Consider $\{X_{T-s}^T\}_{s\in [0,t]}$ starting from $x$ at $s=t$. Since $\tilde{S}_t$ goes to $-\infty$ as $s\rightarrow 0\wedge \tau(x)$. We have $\tilde{S}_t\leq 0$.  Now, setting $a=b=\frac{1}{2}$, we complete the proof.
\end{proof}
\begin{proof}[{\bf Proof of Theorem \ref{cor2}}]
When  $D=B_{\rho,T}$, the function $\varphi$ is chosen  in \eqref{varphi}. Since
  \begin{align}
    c_{\varphi}&=\sup_{B_{\rho,T}}\{3|\nabla \varphi_t|^2-\varphi_t(\Delta -2\partial_t)\varphi_t\}\leq \frac{\pi^2(n+3)}{4\rho^2}+\frac{\pi}{2}(k_1+k_2),\label{const-1}
 \end{align}
 combining this with \eqref{estimate-constant-3} and \eqref{local-li-Yau}, we complete the proof.
\end{proof}

\section{Applications to Ricci flow}
\hspace{0.5cm}
In this section, we apply our results to the special case of the Ricci flow \eqref{Ricci-flow}.
 In this case, when $M$ is compact,
Theorem \ref{Hamilton-1}  reduces to \cite{CH09, Zhang06}. That is
$$\frac{|\nabla u|}{u}\leq \sqrt{\frac{2}{t}\log \frac{\|u\|_{M\times [0,T]}}{u}}.$$
{\bf Remark.}\ When the manifold $M$ carries boundary, we turn to consider the system as follows:
 for $\lambda\geq 0$, \begin{equation}\label{Ricci-flow-boundary}
    \begin{cases}
\frac{\partial}{\partial t}g(x,t)=-{\rm Ric}(x,t), \  & \  \ (x,t)\in M\times [0,T];\\
\mathbb{I}_t=\lambda,\ & \ \ x\in \partial M,
\end{cases}
 \end{equation}
 where $\mathbb{I}_t$ is the second fundamental form w.r.t. the metric $g_t$.
 Shen \cite{Shen} give the proof of the short time exitance of the solution.  Suppose a smooth positive function $u:M\times [0,T]\rightarrow \mathbb{R}$ satisfies the heat equation
\begin{equation}\label{heat-equ-boundary}
\begin{cases}
  \frac{\partial}{\partial t}u(x,t)=\frac{1}{2}\Delta u(x,t),&\ \ (x,t)\in M\times [0,T];\\
  \frac{\partial}{\partial \nu}u(x,t)=0,\ &\ \ x\in \partial M,
  \end{cases}
\end{equation}
where $\frac{\partial}{\partial \nu}$ is the inward unit normal vector field of the boundary associated with $g_t$. Let $X_t^T$ be a $g_{(T-t)}$-Brownian motion and $P_{s,t}$ be the associated semigroup (see \cite{cheng}). Since $\mathbb{I}_t=\lambda\geq 0$ and $\frac{\partial}{\partial \nu}P_{T-t,T}f=0$ (see \cite[Theorem\ 2.1(2)]{cheng}),  the process  $H_t$, constructed as in Lemma \ref{lem2} by using the reflecting $g_{(T-t)}$-Brownian motion,   is also a submartingale for the system \eqref{Ricci-flow-boundary}-\eqref{heat-equ-boundary}. Then, the conclusion in Theorem \ref{Hamilton-1} still holds true. \medskip

Next, we consider the Li-Yau type gradient inequality on a compact manifold carrying the Ricci flow.
\begin{theorem}\label{Ricci-compact}
Suppose the manifold $M$ is compact and $0\leq {\rm Ric}_t\leq k$, $t\in (0,T]$. Let $(M,g_t)_{t\in [0,T]}$ be a solution of the Ricci flow.  Assume a smooth positive function $u:M\times [0,T]\rightarrow \mathbb{R}$ satisfies the heat equation \eqref{heat-equ}. Then the following estimates
\begin{align}\label{compact-e1}
    \frac{|\nabla  u|^2}{u^2}-\frac{\Delta u}{u}\leq kn+\frac{2n}{t}
\end{align}
hold for all $M\times (0,T]$.
\end{theorem}
\begin{proof}
Note that when $\{g_t\}_{t\in [0,T]}$ is a Ricci flow, the second Bianchi identity says that
$${\rm div}(\partial_tg_t)-\frac{1}{2}\nabla ({\rm tr}_g (\partial_tg_t))=0.$$
So we do not need the condition $|\nabla h|\leq k_4$ to deal with $\vd (\Delta u_t)$.  Consider the  process $\hat{S}_t$ constructed in \eqref{s-t}.
 By a similar discussion as in the proof of Theorem  \ref{Li-Yau-gradient-inequality-1}, we obtain \eqref{compact-e1} directly.  Here, we omit the details.
\end{proof}
\noindent{\bf Remark.}
(1)\ In \cite{BCP}, the authors gave the  Li-Yau type gradient estimate on the compact manifold as follows,
   $$ \frac{|\nabla  u|^2}{u^2}-\frac{\Delta u}{u}\leq 2kn+\frac{n}{t}$$ by using the maximal principle, which  has a little difference from  our result. \\
   (2)\ We claim that \eqref{compact-e1} also holds for the system \eqref{Ricci-flow-boundary}-\eqref{heat-equ-boundary}. Since
   $$\frac{\partial}{\partial \nu} \hat{S}_t=h_t \frac{\partial}{\partial \nu}\l(\frac{|\nabla u|^2}{u}- \Delta u\r)=-2 h_t\frac{1}{u^2}\mathbb{I}_t(\nabla u,\nabla u)\leq 0,$$
   where the second equality is ensured by the proof of \cite[Theorem 3.4]{BCP},
   we conclude  that  $\hat{S}_t$ for $c_1=2, c_2=1$ is also a submartingale, which lead to the  assertion announced above.

    \medskip

When the manifold $M$ is noncompact, we present the local version of the gradient estimates under the Ricci flow.
\begin{theorem}\label{Ricci-noncompact}
Suppose $(M,g_t)_{t\in [0,T]}$ is a complete solution of the Ricci flow \eqref{Ricci-flow}. Assume that $|{\rm Ric}_t|\leq k$ for some $k>0$ and all $(x,t)\in B_{\rho,T}$, and a smooth positive function $u:M\times [0,T]\rightarrow \mathbb{R}$ satisfies the heat equation \eqref{heat-equ}. Then for $\alpha>1$ and $t>0$,
\begin{align*}
    & \l|\frac{\nabla  u}{u}\r|^2\leq 2\l(\frac{1}{t}+\frac{4\pi^2(n+7)}{(4-\pi)^2\rho^2}+\frac{8k\pi}{(4-\pi)^2}\r)\l(4+\log \frac{\|u\|_{B_{\rho,T}}}{u}\r),
\end{align*}
and
\begin{equation*}
    \frac{|\nabla  u|^2}{u^2}-\alpha\frac{\Delta  u}{u}\leq \frac{2\alpha^2n}{t}+\frac{8\alpha^2 n\pi^2}{(4-\pi)^2}\frac{n(1+{\alpha^2}{(\alpha-1)^{-1}})+3}{\rho^2}+\l(\frac{4+\pi}{4-\pi}\r)^2\alpha^2 kn+\frac{\alpha^3 kn}{\alpha-1}
\end{equation*}
hold on $B_{\frac{\rho}{2},T}$.
\end{theorem}
\begin{proof}
Under the Ricci flow, $\mathcal{R}_t\equiv 0$, then \eqref{ineq-Li-Yau} holds for $k_1+k_2=0$. Moreover,
since ${\rm Ric}_t=-\partial_tg_t\leq k$, we have
\begin{align}\label{eq5}
(\Delta -2\partial_t)\rho_t(x_0,\cdot)(x)\leq \frac{n-1}{\rho_t(x_0,x)}+k\rho_t(x_0,x),
\end{align}
which implies the first inequity. On the other hand,
 by the curvature condition, we see that the inequity \eqref{local-li-Yau} holds for $k_1=k_2=k_3=k$ and $k_4=0$. Thus, according to \eqref{eq5}  and
Lemma \ref{Li-Yau-gradient-inequality},
  we complete the proof.
\end{proof}
\bigskip
\bigskip

\noindent\textbf{Acknowledgements}  \ The author  thanks Professor Feng-Yu Wang for his guidance. This work is supported in part by 985 Project,
 973 Project.

\end{document}